\documentclass{amsart}
\usepackage{amstext}
\usepackage{amsthm}
\usepackage{amssymb}

\makeatletter
\numberwithin{equation}{section}
\numberwithin{figure}{section}
\theoremstyle{plain}
\newtheorem{thm}{Theorem}
  \theoremstyle{plain}
  \newtheorem{lem}[thm]{Lemma}

\allowdisplaybreaks
\usepackage[sort&compress,square,numbers]{natbib}

\makeatother

\begin{document}

\title{Differential equations associated with $\lambda$-Changhee polynomials }

\author{Taekyun Kim}
\address{Department of Mathematics, Kwangwoon University, Seoul 139-701, Republic
of Korea}
\email{tkkim@kw.ac.kr}

\author{Dae San Kim}
\address{Department of Mathematics, Sogang University, Seoul 121-742, Republic
of Korea}
\email{dskim@sogang.ac.kr}

\begin{abstract}
In this paper, we study linear differential equations arising from
$\lambda$-Changhee polynomials (or called degenerate Changhee polynomials)
and give some explicit and new identities for the $\lambda$-Changhee
polynomials associated with linear differential equations.
\end{abstract}

 \medskip

\maketitle
\global\long\def\relphantom#1{\mathrel{\phantom{{#1}}}}
\global\long\def\Ch{\mathrm{Ch}}

\section{Introduction}

For $N\in\mathbb{N}$, we define the generalized harmonic numbers as
follows:

\begin{equation}
H_{N,0}=1,\quad\text{ for all }N,\label{eq:1}
\end{equation}

\begin{equation}
H_{N,1}=H_{N}=1+\frac{1}{2}+\cdots+\frac{1}{N},\quad\left(\text{see \cite{key-5}}\right),\label{eq:2}
\end{equation}

\begin{equation}
H_{N,j}=\frac{H_{N-1,j-1}}{N}+\frac{H_{N-2,j-1}}{N-1}+\cdots+\frac{H_{j-1,j-1}}{j},\quad\left(2\le j\le N\right).\label{eq:3}
\end{equation}

For $k\in\mathbb{N}$ and $N,j\in\mathbb{N}\cup\left\{ 0\right\} $,
we define the generalized Changhee power sums $S_{k,j}\left(N\right)$
as follows:
\begin{align}
S_{k,0}\left(N\right) & =\left(N+1\right)^{k},\label{eq:4}\\
S_{k,j}\left(N\right) & =\sum_{l=0}^{N}S_{k,j-1}\left(l\right),\quad\left(j\ge1\right),\quad\left(\text{see \cite{key-5,key-6}}\right).\label{eq:5}
\end{align}

In particular, for $k=1$, we also define $S_{1,-1}\left(N\right)=1$.

As is well known, the Euler polynomials are defined by the generating
function
\begin{equation}
\frac{2}{e^{t}+1}e^{xt}=\sum_{n=0}^{\infty}E_{n}\left(x\right)\frac{t^{n}}{n!},\quad\left(\text{see \cite{key-1,key-9,key-10,key-11}}\right).\label{eq:6}
\end{equation}

With the viewpoint of deformed Euler polynomials, the Changhee polynomials
are defined by the generating function
\begin{equation}
\frac{2}{t+2}\left(t+1\right)^{x}=\sum_{n=0}^{\infty}\Ch_{n}\left(x\right)\frac{t^{n}}{n!},\quad\left(\text{see \cite{key-2,key-3,key-8}}\right).\label{eq:7}
\end{equation}

From (\ref{eq:6}) and (\ref{eq:7}), we note that
\begin{align}
 & \sum_{n=0}^{\infty}E_{n}\left(x\right)\frac{t^{n}}{n!}\label{eq:8}\\
 & =\sum_{m=0}^{\infty}\Ch_{m}\left(x\right)\frac{1}{m!}\left(e^{t}-1\right)^{m}\nonumber \\
 & =\sum_{n=0}^{\infty}\left(\sum_{m=0}^{n}\Ch_{m}\left(x\right)S_{2}\left(n,m\right)\right)\frac{t^{n}}{n!},\nonumber
\end{align}
where $S_{2}\left(n,m\right)$ are the Stirling numbers of the second
kind. Thus, by (\ref{eq:8}), we get
\begin{equation}
E_{n}\left(x\right)=\sum_{m=0}^{n}\Ch_{m}\left(x\right)S_{2}\left(n,m\right)\quad\left(n\ge0\right).\label{eq:9}
\end{equation}

The Stirling numbers of the first kind $S_{1}\left(n,l\right)$ appear in the expansion of the falling factorial
\begin{align}
\left(x\right)_{0} & =1,\quad\left(x\right)_{n}=x\left(x-1\right)\cdots\left(x-n+1\right)\label{eq:10}\\
 & =\sum_{l=0}^{n}S_{1}\left(n,l\right)x^{l},\quad\left(n\ge1\right).\nonumber
\end{align}

From (\ref{eq:10}), we note that the generating function of the Stirling
numbers of the first kind is given by
\begin{equation}
\left(\log\left(1+t\right)\right)^{n}=n!\sum_{m=n}^{\infty}S_{1}\left(m,n\right)\frac{t^{m}}{m!},\quad\left(\text{see \cite{key-4,key-7,key-12}}\right).\label{eq:11}
\end{equation}

By (\ref{eq:2}), we easily get
\begin{align}
 & \sum_{n=0}^{\infty}\Ch_{n}\left(x\right)\frac{t^{n}}{n!}\label{eq:12}\\
 & =\sum_{m=0}^{\infty}E_{m}\left(x\right)\frac{1}{m!}\left(\log\left(1+t\right)\right)^{m}\nonumber \\
 & =\sum_{n=0}^{\infty}\left(\sum_{m=0}^{n}S_{1}\left(n,m\right)E_{m}\left(x\right)\right)\frac{t^{m}}{m!}.\nonumber
\end{align}

Thus, by (\ref{eq:12}), we have
\[
\Ch_{n}\left(x\right)=\sum_{m=0}^{n}S_{1}\left(n,m\right)E_{m}\left(x\right),\quad\left(n\ge0\right),\quad\left(\text{see \cite{key-16,key-17}}\right).
\]

Recently, $\lambda$-Changhee polynomials (or called degenerate Changhee
polynomials) are defined by the generating function
\begin{align}
 & \frac{2\lambda}{2\lambda+\log\left(1+\lambda t\right)}\left(1+\frac{\log\left(1+\lambda t\right)}{\lambda}\right)^{x}\label{eq:14}\\
 & =\sum_{n=0}^{\infty}\Ch_{n,\lambda}\left(x\right)\frac{t^{n}}{n!},\quad\left(\text{see \cite{key-13}}\right).\nonumber
\end{align}

When $x=0$, $\Ch_{n,\lambda}=\Ch_{n,\lambda}\left(0\right)$ are
called $\lambda$-Changhee numbers (or called degenerate Changhee
numbers).

In \cite{key-5}, Kim-Kim gave some explicit and new identities for
the Bernoulli numbers of the second kind arising from nonlinear differential
equations. It is known that some interesting identities and properties
of the Frobenius-Euler polynomials are also derived from the non-linear
differential equations (see \cite{key-6,key-11}).

Recently, several authors have studied some interesting properties
for the Changhee numbers and polynomials (see \cite{key-1,key-2,key-3,key-4,key-5,key-6,key-7,key-8,key-9,key-10,key-11,key-12,key-13,key-14,key-15,key-16,key-17,key-18, key-19}).

In this paper, we develop some new method for obtaining identities
related to $\lambda$-Changhee polynomials arising from linear differential
equations. From our study, we derive some explicit and new identities
for the $\lambda$-Changhee polynomials.

\section{Some identities for the $\lambda$-Changhee polynomials arising from
linear differential equations}

First, we introduce lemma for the generalized Changhee power sum $S_{k,j}\left(N\right)$.
\begin{lem}
\label{lem:1} For $2\le r\le N$ and $1\le i\le r-1$, we have
\begin{equation}
S_{1,i-1}\left(r-1-i\right)+S_{1,i-2}\left(r-i\right)=S_{1,i-1}\left(r-i\right).\tag{{A}}\label{eq:lem-1-a}
\end{equation}
\end{lem}
\begin{proof}
From (\ref{eq:4}) and (\ref{eq:5}), we have
\begin{align}
S_{1,i-1}\left(r-1-i\right)+S_{1,i-2}\left(r-i\right) & =\sum_{l=0}^{r-1-i}S_{1,i-2}\left(l\right)+S_{1,i-2}\left(r-i\right)\label{eq:16}\\
 & =\sum_{l=0}^{r-i}S_{1,i-2}\left(l\right)\nonumber \\
 & =S_{1,i-1}\left(r-i\right).\nonumber \qedhere
\end{align}
\end{proof}
Let
\begin{align}
F & =F\left(t;x,\lambda\right)\label{eq:17}\\
 & =\frac{2\lambda}{2\lambda+\log\left(1+\lambda t\right)}\left(1+\lambda^{-1}\log\left(1+\lambda t\right)\right)^{x}.\nonumber
\end{align}

Then, by (\ref{eq:17}), we get
\begin{align}
 & F^{\left(1\right)}\label{eq:18}\\
 & =\frac{d}{dt}F\left(t;x,\lambda\right)\nonumber \\
 & =\lambda\left(1+\lambda t\right)^{-1}\left(-\left(2\lambda+\log\left(1+\lambda t\right)\right)^{-1}+x\left(\lambda+\log\left(1+\lambda t\right)\right)^{-1}\right)F.\nonumber
\end{align}
\begin{align}
 & F^{\left(2\right)}\label{eq:19}\\
 & =\frac{dF^{\left(1\right)}}{dt}\nonumber \\
 & =\lambda^{2}\left(1+\lambda t\right)^{-2}\left\{ \left(2\lambda+\log\left(1+\lambda t\right)\right)^{-1}-x\left(\lambda+\log\left(1+\lambda t\right)\right)^{-1}\right.\nonumber \\
 & \relphantom =+2\left(2\lambda+\log\left(1+\lambda t\right)\right)^{-2}-2x\left(2\lambda+\log\left(1+\lambda t\right)\right)^{-1}\left(\lambda+\log\left(1+\lambda t\right)\right)^{-1}\nonumber \\
 & \relphantom =\left.+\left(x\right)_{2}\left(\lambda+\log\left(1+\lambda t\right)\right)^{-2}\right\} F.\nonumber
\end{align}

So, we are led to put
\begin{align}
 & F^{\left(N\right)}\label{eq:20}\\
 & =\left(\frac{d}{dt}\right)^{N}F\left(t;x,\lambda\right)\nonumber \\
 & =\lambda^{N}\left(1+\lambda t\right)^{-N}\left(\sum_{1\le i+j\le N}a_{i,j}^{\left(\lambda\right)}\left(N,x\right)\left(2\lambda+\log\left(1+\lambda t\right)\right)^{-i}\left(\lambda+\log\left(1+\lambda t\right)\right)^{-j}\right)F,\nonumber
\end{align}
where $N=1,2,\dots$, and the sum is over all nonnegative integers
$i,j$ with $1\le i+j\le N$.

On the one hand, by (\ref{eq:20}), we get
\begin{align}
 &\relphantom{=} F^{\left(N+1\right)}\label{eq:21}\\
 & =\frac{dF^{\left(N\right)}}{dt}\nonumber \\
 & =\lambda^{N+1}\left(1+\lambda t\right)^{-\left(N+1\right)}\nonumber \\
 & \relphantom =\times\left\{ \left(-N\right)\sum_{\substack{1\le i+j\le N\\i\ge0,j\ge0}}a_{i,j}^{\left(\lambda\right)}\left(N,x\right)\left(2\lambda+\log\left(1+\lambda t\right)\right)^{-i}\left(\lambda+\log\left(1+\lambda t\right)\right)^{-j}\right.\nonumber \\
 & \relphantom =-\sum_{\substack{2\le i+j\le N+1\\
i\ge1,\,j\ge0
}
}ia_{i-1,j}^{\left(\lambda\right)}\left(N,x\right)\left(2\lambda+\log\left(1+\lambda t\right)\right)^{-i}\left(\lambda+\log\left(1+\lambda t\right)\right)^{-j}\nonumber \\
 & \relphantom =\left.+\sum_{\substack{2\le i+j\le N+1\\
i\ge0,j\ge1
}
}\left(x-j+1\right)a_{i,j-1}^{\left(\lambda\right)}\left(N,x\right)\left(2\lambda+\log\left(1+\lambda t\right)\right)^{-i}\nonumber\right.\\
&\relphantom{=}\times\left.\left(\lambda+\log\left(1+\lambda t\right)\right)^{-j}\right\} F.\nonumber
\end{align}

On the other hand, by replacing $N$ by $N+1$ in (\ref{eq:20}),
we have
\begin{align}\label{eq:22}
F^{\left(N+1\right)} & =\lambda^{N+1}\left(1+\lambda t\right)^{-\left(N+1\right)} \\
 & \relphantom =\times\left(\sum_{\substack{1\le i+j\le N+1\\
i,j\ge0
}
}a_{i,j}^{\left(\lambda\right)}\left(N+1,x\right)\left(2\lambda+\log\left(1+\lambda t\right)\right)^{-i}\right.\nonumber\\
&\relphantom{===}\times\left.\left(\lambda+\log\left(1+\lambda t\right)\right)^{-j}\right)F.\nonumber\\
\end{align}

Let $i+j=r$. Then $1\le r\le N+1$. Comparing the terms with $r=1$,
we get
\begin{align}
a_{1,0}^{\left(\lambda\right)}\left(N+1,x\right) & =-Na_{1,0}^{\left(\lambda\right)}\left(N,x\right)\label{eq:23}\\
a_{0,1}^{\left(\lambda\right)}\left(N+1,x\right) & =-Na_{0,1}^{\left(\lambda\right)}\left(N,x\right).\nonumber
\end{align}

Comparing the terms with $i+j=r$ $\left(2\le r\le N\right)$,
\begin{align}
 & \sum_{i=0}^{r}a_{i,r-i}^{\left(\lambda\right)}\left(N+1,x\right)\left(2\lambda+\log\left(1+\lambda t\right)\right)^{-i}\left(\lambda+\log\left(1+\lambda t\right)\right)^{i-r}\label{eq:24}\\
 & =-N\sum_{i=0}^{r}a_{i,r-i}^{\left(\lambda\right)}\left(N,x\right)\left(2\lambda+\log\left(1+\lambda t\right)\right)^{-i}\left(\lambda+\log\left(1+\lambda t\right)\right)^{i-r}\nonumber \\
 & \relphantom =-\sum_{i=1}^{r}ia_{i-1,r-i}^{\left(\lambda\right)}\left(N,x\right)\left(2\lambda+\log\left(1+\lambda t\right)\right)^{-i}\left(\lambda+\log\left(1+\lambda t\right)\right)^{i-r}\nonumber \\
 & \relphantom =+\sum_{i=0}^{r-1}\left(x+i-r+1\right)a_{i,r-i-1}^{\left(\lambda\right)}\left(N,x\right)\left(2\lambda+\log\left(1+\lambda t\right)\right)^{-i}\left(\lambda+\log\left(1+\lambda t\right)\right)^{i-r}.\nonumber
\end{align}

Thus, by (\ref{eq:24}), we get
\begin{align}
a_{0,r}^{\left(\lambda\right)}\left(N+1,x\right) & =-Na_{0,r}^{\left(\lambda\right)}\left(N,x\right)+\left(x-r+1\right)a_{0,r-1}^{\left(\lambda\right)}\left(N,x\right),\label{eq:25}\\
a_{r,0}^{\left(\lambda\right)}\left(N+1,x\right) & =-Na_{r,0}^{\left(\lambda\right)}\left(N,x\right)-ra_{r-1,0}^{\left(\lambda\right)}\left(N,x\right),\label{eq:26}
\end{align}
and
\begin{align}\label{eq:27}
&\relphantom{=}a_{i,r-i}^{\left(\lambda\right)}\left(N+1,x\right)\\
&=-Na_{i,r-i}^{\left(\lambda\right)}\left(N,x\right)-ia_{i-1,r-i}^{\left(\lambda\right)}\left(N,x\right)\nonumber\\
&\relphantom{=}+\left(x+i-r+1\right)a_{i,r-i-1}^{\left(\lambda\right)}\left(N,x\right),\nonumber
\end{align}
where $1\le i\le r-1$.

Comparing the terms with $i+j=N+1$, we get
\begin{align}
 & \sum_{i=0}^{N+1}a_{i,N+1-i}^{\left(\lambda\right)}\left(N+1,x\right)\left(2\lambda+\log\left(1+\lambda t\right)\right)^{-i}\left(\lambda+\log\left(1+\lambda t\right)\right)^{i-\left(N+1\right)}\label{eq:28}\\
 & =-\sum_{i=1}^{N+1}ia_{i-1,N+1-i}^{\left(\lambda\right)}\left(N,x\right)\left(2\lambda+\log\left(1+\lambda t\right)\right)^{-i}\left(\lambda+\log\left(1+\lambda t\right)\right)^{i-\left(N+1\right)}\nonumber \\
 & \relphantom =+\sum_{i=0}^{N}\left(x+i-N\right)a_{i,N-i}^{\left(\lambda\right)}\left(N,x\right)\left(2\lambda+\log\left(1+\lambda t\right)\right)^{-i}\left(\lambda+\log\left(1+\lambda t\right)\right)^{i-\left(N+1\right)}.\nonumber
\end{align}

From (\ref{eq:28}), we note that
\begin{align}\label{eq:29}
&a_{0,N+1}^{\left(\lambda\right)}\left(N+1,x\right)=\left(x-N\right)a_{0,N}^{\left(\lambda\right)}\left(N,x\right),\\
&a_{N+1,0}^{\left(\lambda\right)}\left(N+1,x\right)=-\left(N+1\right)a_{N,0}^{\left(\lambda\right)}\left(N,x\right),\nonumber
\end{align}
and
\begin{equation}
a_{i,N+1-i}^{\left(\lambda\right)}\left(N+1,x\right)=-ia_{i-1,N+1-i}^{\left(\lambda\right)}\left(N,x\right)+\left(x+i-N\right)a_{i,N-i}^{\left(\lambda\right)}\left(N,x\right),\label{eq:30}
\end{equation}
where $1\le i\le N$.

From (\ref{eq:18}) and (\ref{eq:20}), we have
\begin{align}
 & \lambda\left(1+\lambda t\right)^{-1}\left(-\left(2\lambda+\log\left(1+\lambda t\right)\right)^{-1}+x\left(\lambda+\log\left(1+\lambda t\right)\right)^{-1}\right)F\label{eq:31}\\
 & =F^{\left(1\right)}\nonumber \\
 & =\lambda\left(1+\lambda t\right)^{-1}\left(a_{1,0}^{\left(\lambda\right)}\left(1,x\right)\left(2\lambda+\log\left(1+\lambda t\right)\right)^{-1}\right.\nonumber\\
 &\relphantom{=}\left.+a_{0,1}^{\left(\lambda\right)}\left(1,x\right)\left(\lambda+\log\left(1+\lambda t\right)\right)^{-1}\right)F.\nonumber
\end{align}

By comparing the coefficients on both sides of (\ref{eq:31}),
we have
\begin{equation}
a_{1,0}^{\left(\lambda\right)}\left(1,x\right)=-1,\quad a_{0,1}^{\left(\lambda\right)}\left(1,x\right)=x.\label{eq:32}
\end{equation}

From (\ref{eq:23}), we note that
\begin{equation}
a_{1,0}^{\left(\lambda\right)}\left(N+1,x\right)=\left(-1\right)^{N+1}N!,\quad a_{0,1}^{\left(\lambda\right)}\left(N+1,x\right)=\left(-1\right)^{N}N!x.\label{eq:33}
\end{equation}

By (\ref{eq:29}), we easily get
\begin{equation}
a_{N+1,0}^{\left(\lambda\right)}\left(N+1,x\right)=\left(-1\right)^{N+1}\left(N+1\right)!,\quad a_{0,N+1}^{\left(\lambda\right)}\left(N+1,x\right)=\left(x\right)_{N+1}.\label{eq:34}
\end{equation}

For $i=1$ in (\ref{eq:30}), we have
\begin{align}
 & a_{1,N}^{\left(\lambda\right)}\left(N+1,x\right)\label{eq:35}\\
 & =-a_{0,N}^{\left(\lambda\right)}\left(N,x\right)+\left(x+1-N\right)a_{1,N-1}^{\left(\lambda\right)}\left(N,x\right)\nonumber \\
 & =-\left(x\right)_{N}+\left(x+1-N\right)a_{1,N-1}^{\left(\lambda\right)}\left(N,x\right)\nonumber \\
 & =-\left(x\right)_{N}+\left(x+1-N\right)\left(-\left(x\right)_{N-1}+\left(x+2-N\right)a_{1,N-2}^{\left(\lambda\right)}\left(N-1,x\right)\right)\nonumber \\
 & =-2\left(x\right)_{N}+\left(x+1-N\right)\left(x+2-N\right)a_{1,N-2}^{\left(\lambda\right)}\left(N-1,x\right)\nonumber \\
 & \vdots\nonumber \\
 & =-N\left(x\right)_{N}+\left(x+1-N\right)\left(x+2-N\right)\cdots\left(x+N-N\right)a_{1,0}^{\left(\lambda\right)}\left(1,x\right)\nonumber \\
 & =-\left(N+1\right)\left(x\right)_{N}\nonumber \\
 & =-S_{1,0}\left(N\right)\left(x\right)_{N}.\nonumber
\end{align}

For $i=2$ in (\ref{eq:30}), we note that
\begin{align}
&\relphantom{=}a_{2,N-1}^{\left(\lambda\right)}\left(N+1,x\right) \label{eq:36}\\
& =-2a_{1,N-1}^{\left(\lambda\right)}\left(N,x\right)\nonumber\\
 & \relphantom{=} +\left(x+2-N\right)a_{2,N-2}^{\left(\lambda\right)}\left(N,x\right)\nonumber \\
 & =\left(-1\right)^{2}2N\left(x\right)_{N-1}\nonumber\\
 &\relphantom{=}+\left(x+2-N\right)\left\{ \left(-1\right)^{2}2\left(N-1\right)\left(x\right)_{N-2}+\left(x+3-N\right)a_{2,N-3}^{\left(\lambda\right)}\left(N-1,x\right)\right\} \nonumber \\
 & =\left(-1\right)^{2}2\left\{ N+\left(N-1\right)\right\} \left(x\right)_{N-1}+\left(x+2-N\right)\left(x+3-N\right)a_{2,N-3}^{\left(\lambda\right)}\left(N-1,x\right)\nonumber \\
 & \vdots\nonumber \\
 & =\left(-1\right)^{2}2\left\{ N+\left(N-1\right)+\left(N-2\right)+\cdots+2\right\} \left(x\right)_{N-1}\nonumber\\
 &\relphantom{=}+\left(x+2-N\right)\left(x+3-N\right)\cdots xa_{2,0}^{\left(\lambda\right)}\left(2,x\right)\nonumber \\
 & =\left(-1\right)^{2}2\left\{ N+\left(N-1\right)+\cdots+2+1\right\} \left(x\right)_{N-1}\nonumber \\
 & =\left(-1\right)^{2}2!S_{1,1}\left(N-1\right)\left(x\right)_{N-1}.\nonumber
\end{align}

Let $i=3$ in (\ref{eq:30}). Then we have
\begin{align}
 & a_{3,N-2}^{\left(\lambda\right)}\left(N+1,x\right)\label{eq:37}\\
 & =-3a_{2,N-2}^{\left(\lambda\right)}\left(N,x\right)+\left(x+3-N\right)a_{3,N-3}^{\left(\lambda\right)}\left(N,x\right)\nonumber \\
 & =\left(-1\right)^{3}3!S_{1,1}\left(N-2\right)\left(x\right)_{N-2}+\left(x+3-N\right)a_{3,N-3}^{\left(\lambda\right)}\left(N,x\right)\nonumber \\
 & \vdots\nonumber \\
 & =\left(-1\right)^{3}3!\left\{ S_{1,1}\left(N-2\right)+\cdots+S_{1,1}\left(1\right)\right\} \left(x\right)_{N-2}\nonumber\\
 &\relphantom{=}+\left(x+3-N\right)\left(x+4-N\right)\cdots xa_{3,0}^{\left(\lambda\right)}\left(3,x\right)\nonumber \\
 & =\left(-1\right)^{3}3!\left\{ S_{1,1}\left(N-2\right)+\cdots+S_{1,1}\left(1\right)+S_{1,1}\left(0\right)\right\} \left(x\right)_{N-2}\nonumber \\
 & =\left(-1\right)^{3}3!S_{1,2}\left(N-2\right)\left(x\right)_{N-2}.\nonumber
\end{align}

Continuing this process, we get
\begin{equation}
a_{i,N+1-i}^{\left(\lambda\right)}\left(N+1,x\right)=\left(-1\right)^{i}i!S_{1,i-1}\left(N-i+1\right)\left(x\right)_{N-i+1},\quad\left(1\le i\le N\right).\label{eq:38}
\end{equation}

Let $2\le r\le N$. Then, by (\ref{eq:25}), we get
\begin{align}
a_{0,r}^{\left(\lambda\right)}\left(N+1,x\right) & =\left(x-r+1\right)a_{0,r-1}^{\left(\lambda\right)}\left(N,x\right)-Na_{0,r}^{\left(\lambda\right)}\left(N,x\right)\label{eq:39}\\
 & =\left(x-r+1\right)a_{0,r-1}^{\left(\lambda\right)}\left(N,x\right)\nonumber\\
 &\relphantom{=}-N\left\{ \left(x-r+1\right)a_{0,r-1}^{\left(\lambda\right)}\left(N-1,x\right)-\left(N-1\right)a_{0,r}^{\left(\lambda\right)}\left(N-1,x\right)\right\} \nonumber \\
 & =\left(x-r+1\right)\left\{ a_{0,r-1}^{\left(\lambda\right)}\left(N,x\right)-Na_{0,r-1}^{\left(\lambda\right)}\left(N-1,x\right)\right\}\nonumber\\
 &\relphantom{=} +\left(-1\right)^{2}N\left(N-1\right)a_{0,r}^{\left(\lambda\right)}\left(N-1,x\right)\nonumber \\
 & \vdots\nonumber \\
 & =\left(x-r+1\right)\sum_{i=0}^{N-r}\left(-1\right)^{i}\left(N\right)_{i}a_{0,r-1}^{\left(\lambda\right)}\left(N-i,x\right)\nonumber\\
 &\relphantom{=}+\left(-1\right)^{N-r+1}N\left(N-1\right)\cdots ra_{0,r}^{\left(\lambda\right)}\left(r,x\right)\nonumber \\
 & =\left(x-r+1\right)\sum_{i=0}^{N-r+1}\left(-1\right)^{i}\left(N\right)_{i}a_{0,r-1}^{\left(\lambda\right)}\left(N-i,x\right).\nonumber
\end{align}

Now, we give an explicit expression for $a_{0,r}^{\left(\lambda\right)}\left(N+1,x\right)$
$\left(2\le r\le N\right)$.

For $r=2$ in (\ref{eq:39}), we have
\begin{align}
a_{0,2}^{\left(\lambda\right)}\left(N+1,x\right) & =\left(x-1\right)\sum_{i=0}^{N-1}\left(-1\right)^{i}\left(N\right)_{i}a_{0,1}^{\left(\lambda\right)}\left(N-i,x\right)\label{eq:40}\\
 & =\left(x\right)_{2}\left(-1\right)^{N-1}\sum_{i=0}^{N-1}\left(N\right)_{i}\left(N-i-1\right)!\nonumber \\
 & =\left(x\right)_{2}\left(-1\right)^{N-1}N!\sum_{i=0}^{N-1}\frac{1}{N-i}\nonumber \\
 & =\left(x\right)_{2}\left(-1\right)^{N-1}N!\sum_{i=1}^{N}\frac{1}{i}\nonumber \\
 & =\left(x\right)_{2}\left(-1\right)^{N-1}N!H_{N,1}.\nonumber
\end{align}

Let us consider $r=3$ in (\ref{eq:39}). From (\ref{eq:40}), we
note that
\begin{align}
a_{0,3}^{\left(\lambda\right)}\left(N+1,x\right) & =\left(x-2\right)\sum_{i=0}^{N-2}\left(-1\right)^{i}\left(N\right)_{i}a_{0,2}^{\left(\lambda\right)}\left(N-i,x\right)\label{eq:41}\\
 & =\left(x-2\right)\sum_{i=0}^{N-2}\left(-1\right)^{i}\left(N\right)_{i}\left(x\right)_{2}\left(-1\right)^{N-i-2}\left(N-i-1\right)!H_{N-i-1}\nonumber \\
 & =\left(x\right)_{3}\left(-1\right)^{N-2}N!\sum_{i=0}^{N-2}\frac{H_{N-i-1}}{N-i}\nonumber \\
 & =\left(x\right)_{3}\left(-1\right)^{N-2}N!H_{N,2}.\nonumber
\end{align}

For $r=4$ in (\ref{eq:39}), we have
\begin{align}
 & a_{0,4}^{\left(\lambda\right)}\left(N+1,x\right)\label{eq:42}\\
 & =\left(x-3\right)\sum_{i=0}^{N-3}\left(-1\right)^{i}\left(N\right)_{i}a_{0,3}^{\left(\lambda\right)}\left(N-i,x\right)\nonumber \\
 & =\left(x-3\right)\sum_{i=0}^{N-3}\left(-1\right)^{i}\left(N\right)_{i}\left(x\right)_{3}\left(-1\right)^{N-i-3}\left(N-i-1\right)!H_{N-i-1,2}\nonumber \\
 & =\left(x\right)_{4}\left(-1\right)^{N-3}N!\sum_{i=0}^{N-3}\frac{H_{N-i-1,2}}{N-i}\nonumber \\
 & =\left(x\right)_{4}\left(-1\right)^{N-3}N!\left\{ \frac{H_{N-1,2}}{N}+\frac{H_{N-2,2}}{N-1}+\cdots+\frac{H_{2,2}}{3}\right\} \nonumber \\
 & =\left(x\right)_{4}\left(-1\right)^{N-3}N!H_{N,3}.\nonumber
\end{align}

Continuing this process, we get
\begin{equation}
a_{0,r}^{\left(\lambda\right)}\left(N+1,x\right)=\left(x\right)_{r}\left(-1\right)^{N-r+1}N!H_{N,r-1}.\label{eq:43}
\end{equation}

For $2\le r\le N$, by (\ref{eq:26}), we get
\begin{align}
 & a_{r,0}^{\left(\lambda\right)}\left(N+1,x\right)\label{eq:44}\\
 & =-ra_{r-1,0}^{\left(\lambda\right)}\left(N,x\right)-Na_{r,0}^{\left(\lambda\right)}\left(N,x\right)\nonumber \\
 & \vdots\nonumber \\
 & =-r\sum_{i=0}^{N-r}\left(-1\right)^{i}\left(N\right)_{i}a_{r-1,0}^{\left(\lambda\right)}\left(N-i,x\right)+\left(-1\right)^{N-r+1}N\left(N-1\right)\cdots ra_{r,0}^{\left(\lambda\right)}\left(r,x\right)\nonumber \\
 & =-r\sum_{i=0}^{N-r}\left(-1\right)^{i}\left(N\right)_{i}a_{r-1,0}^{\left(\lambda\right)}\left(N-i,x\right)+\left(-1\right)^{N-r+1}\left(N\right)_{N-r+1}\left(-1\right)^{r}r!\nonumber \\
 & =-r\sum_{i=0}^{N-r+1}\left(-1\right)^{i}\left(N\right)_{i}a_{r-1,0}^{\left(\lambda\right)}\left(N-i,x\right).\nonumber
\end{align}

Let $r=2$ in (\ref{eq:44}). Then, we have
\begin{align}
a_{2,0}^{\left(\lambda\right)}\left(N+1,x\right) & =-2\sum_{i=0}^{N-1}\left(-1\right)^{i}\left(N\right)_{i}a_{1,0}^{\left(\lambda\right)}\left(N-i,x\right)\label{eq:45}\\
 & =2\left(-1\right)^{N+1}\sum_{i=0}^{N+1}\left(N\right)_{i}\left(N-i-1\right)!\nonumber \\
 & =2\left(-1\right)^{N+1}N!\sum_{i=0}^{N-1}\frac{1}{N-i}\nonumber \\
 & =2\left(-1\right)^{N+1}N!\sum_{i=1}^{N}\frac{1}{i}\nonumber \\
 & =2\left(-1\right)^{N+1}N!H_{N,1}.\nonumber
\end{align}

For $r=3$ in (\ref{eq:44}), by (\ref{eq:45}), we get
\begin{align}
a_{3,0}^{\left(\lambda\right)}\left(N+1,x\right) & =-3\sum_{i=0}^{N-2}\left(-1\right)^{i}\left(N\right)_{i}a_{2,0}^{\left(\lambda\right)}\left(N-i,x\right)\label{eq:46}\\
 & =-3\sum_{i=0}^{N-2}\left(-1\right)^{i}\left(N\right)_{i}2\left(-1\right)^{N-i}\left(N-i-1\right)!H_{N-i-1,1}\nonumber \\
 & =3!\left(-1\right)^{N+1}N!\sum_{i=0}^{N-2}\frac{H_{N-i-1,1}}{N-i}\nonumber \\
 & =3!\left(-1\right)^{N+1}N!H_{N,2}.\nonumber
\end{align}

From (\ref{eq:46}), by $r=4$ in (\ref{eq:44}), we note that
\begin{align}
a_{4,0}^{\left(\lambda\right)}\left(N+1,x\right) & =-4\sum_{i=0}^{N-3}\left(-1\right)^{i}\left(N\right)_{i}a_{3,0}^{\left(\lambda\right)}\left(N-i,x\right)\label{eq:47}\\
 & =\left(-1\right)^{N+1}4!N!\sum_{i=0}^{N-3}\frac{H_{N-i-1,2}}{N-i}\nonumber \\
 & =\left(-1\right)^{N+1}4!N!H_{N,3}.\nonumber
\end{align}

Continuing this process, we get
\begin{equation}
a_{r,0}^{\left(\lambda\right)}\left(N+1,x\right)=\left(-1\right)^{N+1}N!r!H_{N,r-1},\quad\left(2\le r\le N\right).\label{eq:48}
\end{equation}

Let $2\le r\le N$, $1\le i\le r-1$. Then, by (\ref{eq:27}), we
get
\begin{align}
 & a_{i,r-i}^{\left(\lambda\right)}\left(N+1,x\right)\label{eq:49}\\
 & =\left(x+i-r+1\right)a_{i,r-i-1}^{\left(\lambda\right)}\left(N,x\right)-ia_{i-1,r-i}^{\left(\lambda\right)}\left(N,x\right)-Na_{i,r-i}^{\left(\lambda\right)}\left(N,x\right)\nonumber \\
 & =\left(x+i-r+1\right)\left\{ a_{i,r-i-1}^{\left(\lambda\right)}\left(N,x\right)-Na_{i,r-i-1}^{\left(\lambda\right)}\left(N-1,x\right)\right\} \nonumber \\
 & \relphantom =-i\left\{ a_{i-1,r-i}^{\left(\lambda\right)}\left(N,x\right)-Na_{i-1,r-i}^{\left(\lambda\right)}\left(N-1,x\right)\right\} \nonumber \\
 & \relphantom =+\left(-1\right)^{2}N\left(N-1\right)a_{i,r-i}^{\left(\lambda\right)}\left(N-1,x\right)\nonumber \\
 & \vdots\nonumber \\
 & =\left(x+i-r+1\right)\sum_{s=0}^{N-r}\left(-1\right)^{s}\left(N\right)_{s}a_{i,r-i-1}^{\left(\lambda\right)}\left(N-s,x\right)\nonumber\\
 &\relphantom{=}-i\sum_{s=0}^{N-r}\left(-1\right)^{s}\left(N\right)_{s}a_{i-1,r-i}^{\left(\lambda\right)}\left(N-s,x\right)\nonumber \\
 & \relphantom =+\left(-1\right)^{N-r+1}\left(N\right)_{N-r+1}a_{i,r-i}^{\left(\lambda\right)}\left(r,x\right)\nonumber \\
 & =\left(x+i-r+1\right)\sum_{s=0}^{N-r}\left(-1\right)^{s}\left(N\right)_{s}a_{i,r-i-1}^{\left(\lambda\right)}\left(N-s,x\right)\nonumber\\
 &\relphantom{=}-i\sum_{s=0}^{N-r}\left(-1\right)^{s}\left(N\right)_{s}a_{i-1,r-i}^{\left(\lambda\right)}\left(N-s,x\right)\nonumber \\
 & \relphantom =+\left(-1\right)^{N-r+1}\left(N\right)_{N-r+1}\left(-1\right)^{i}i!S_{1,i-1}\left(r-i\right)\left(x\right)_{r-i}.\nonumber
\end{align}

Let $r=2$ in (\ref{eq:49}). Then $i=1$. From (\ref{eq:49}), we
note that
\begin{align}
 & a_{1,1}^{\left(\lambda\right)}\left(N+1,x\right)\label{eq:50}\\
 & =x\sum_{s=0}^{N-2}\left(-1\right)^{s}\left(N\right)_{s}a_{1,0}^{\left(\lambda\right)}\left(N-s,x\right)\nonumber\\
 &\relphantom{=}-\sum_{s=0}^{N-2}\left(-1\right)^{s}\left(N\right)_{s}a_{0,1}^{\left(\lambda\right)}\left(N-s,x\right)\nonumber \\
 & \relphantom =+\left(-1\right)^{N-1}N!\left(-1\right)S_{1,0}\left(1\right)x\nonumber \\
 & =x\sum_{s=0}^{N-2}\left(-1\right)^{s}\left(N\right)_{s}\left(-1\right)^{N-s}\left(N-s-1\right)!\nonumber\\
 &\relphantom{=}-\sum_{s=0}^{N-2}\left(-1\right)^{s}\left(N\right)_{s}\left(-1\right)^{N-s-1}\left(N-s-1\right)!x\nonumber \\
 & \relphantom =+2\left(-1\right)^{N}N!x\nonumber \\
 & =2x\left(-1\right)^{N}\sum_{s=0}^{N-2}\left(N\right)_{s}\left(N-s-1\right)!+2\left(-1\right)^{N}N!x\nonumber \\
 & =2x\left(-1\right)^{N}N!\sum_{s=0}^{N-2}\frac{1}{N-s}+2\left(-1\right)^{N}N!x\nonumber \\
 & =2\left(-1\right)^{N}N!x\left\{ \frac{1}{N}+\frac{1}{N-1}+\cdots+\frac{1}{2}+\frac{1}{1}\right\} \nonumber \\
 & =2\left(-1\right)^{N}N!xH_{N}.\nonumber
\end{align}

Let $r=3$. Then $1\le i\le2.$ From (\ref{eq:49}), we note that

\begin{align}
a_{i,3-i}^{\left(\lambda\right)}\left(N+1,x\right) & =\left(x+i-2\right)\sum_{s=0}^{N-3}\left(-1\right)^{s}\left(N\right)_{s}a_{i,2-i}^{\left(\lambda\right)}\left(N-s,x\right)\label{eq:51}\\
 & \relphantom{=}-i\sum_{s=0}^{N-3}\left(-1\right)^{s}\left(N\right)_{s}a_{i-1,3-i}^{\left(\lambda\right)}\left(N-s,x\right)\nonumber \\
 & \relphantom{=}+\left(-1\right)^{N-2}\left(N\right)_{N-2}\left(-1\right)^{i}i!S_{1,i-1}\left(3-i\right)\left(x\right)_{3-i}.\nonumber
\end{align}

For $i=1$ in (\ref{eq:51}), we have
\begin{align}
 & a_{1,2}^{\left(\lambda\right)}\left(N+1,x\right)\label{eq:52}\\
 & =\left(x-1\right)\sum_{s=0}^{N-3}\left(-1\right)^{s}\left(N\right)_{s}a_{1,1}^{\left(\lambda\right)}\left(N-s,x\right)\nonumber \\
 & \relphantom{=}-\sum_{s=0}^{N-3}\left(-1\right)^{s}\left(N\right)_{s}a_{0,2}^{\left(\lambda\right)}\left(N-s,x\right)\nonumber \\
 & \relphantom{=}+\left(-1\right)^{N-2}\left(N\right)_{N-2}\left(-1\right)S_{1,0}\left(2\right)\left(x\right)_{2}\nonumber \\
 & =\left(x-1\right)\sum_{s=0}^{N-3}\left(-1\right)^{s}\left(N\right)_{s}2\left(-1\right)^{N-s-1}\left(N-s-1\right)!H_{N-s-1,1}x\nonumber \\
 & \relphantom{=}-\sum_{s=0}^{N-3}\left(-1\right)^{s}\left(N\right)_{s}\left(x\right)_{2}\left(-1\right)^{N-s-2}\left(N-s-1\right)!H_{N-s-1,1}\nonumber \\
 & \relphantom{=}+3\left(-1\right)^{N-1}\left(N\right)_{N-2}\left(x\right)_{2}\nonumber \\
 & =3\left(x\right)_{2}\left(-1\right)^{N-1}N!\sum_{s=0}^{N-3}\frac{H_{N-s-1,1}}{N-s}+3\left(x\right)_{2}\left(-1\right)^{N-1}N!\frac{1}{2}\nonumber \\
 & =3\left(x\right)_{2}\left(-1\right)^{N-1}N!\left\{ \frac{H_{N-1,1}}{N}+\frac{H_{N-2,1}}{N-1}+\cdots+\frac{H_{2,1}}{3}+\frac{H_{1,1}}{2}\right\} \nonumber \\
 & =3\left(x\right)_{2}\left(-1\right)^{N-1}N!H_{N,2}.\nonumber
\end{align}

Let $i=2$ in (\ref{eq:51}). From (\ref{eq:52}), we note that

\begin{align*}
a_{2,1}^{\left(\lambda\right)}\left(N+1,x\right) & =x\sum_{s=0}^{N-3}\left(-1\right)^{s}\left(N\right)_{s}a_{2,0}^{\left(\lambda\right)}\left(N-s,x\right)\\
&\relphantom{=}-2\sum_{s=0}^{N-3}\left(-1\right)^{s}\left(N\right)_{s}a_{1,1}^{\left(\lambda\right)}\left(N-s,x\right)\\
 & \relphantom =+\left(-1\right)^{N-2}\left(N\right)_{N-2}\left(-1\right)^{2}2!S_{1,1}\left(1\right)x\\
 & =x\sum_{s=0}^{N-3}\left(-1\right)^{s}\left(N\right)_{s}2\left(-1\right)^{N-s}\left(N-s-1\right)!H_{N-s-1,1}\\
 & \relphantom =-2\sum_{s=0}^{N-3}\left(-1\right)^{s}\left(N\right)_{s}2\left(-1\right)^{N-s-1}\left(N-s-1\right)!H_{N-s-1,1}x\\
 & \relphantom =+\left(-1\right)^{N-2}\left(N\right)_{N-2}\left(-1\right)^{2}3!x\\
 & =6x\left(-1\right)^{N}N!\sum_{s=0}^{N-3}\frac{H_{N-s-1,1}}{N-s}+6x\left(-1\right)^{N}\left(N\right)_{N-2}\\
 & =6x\left(-1\right)^{N}N!\left\{ \frac{H_{N-1,1}}{N}+\cdots+\frac{H_{2,1}}{3}+\frac{H_{1,1}}{2}\right\} \\
 & =6x\left(-1\right)^{N}N!H_{N,2}.
\end{align*}

Therefore, we obtain the following theorem.
\begin{thm}
\label{thm:2} Let $1\le r\le N+1$, $0\le i\le r$. Then we have
\[
a_{i,r-i}^{\left(\lambda\right)}\left(N+1,x\right)=\left(-1\right)^{N+1+i-r}i!S_{1,i-1}\left(r-i\right)N!H_{N,r-1}\left(x\right)_{r-i}.
\]
\end{thm}
\begin{proof}
We showed that it is true for $r=1$ and $r=N+1$. Assume that $2\le r\le N$.
If $i=0$ or $i=r$, then it is also true. So we prove the assertion
by inductino on $r$ when $2\le r\le N$, $1\le i\le r-1$. For $r=2$,
$i=1$, we showed that
\[
a_{1,1}^{\left(\lambda\right)}\left(N+1,x\right)=2\left(-1\right)^{N}N!H_{N,1}x.
\]

Assume now that it is true for $r-1\left(3\le r\le N\right)$.

From \ref{eq:49}, we note that
\begin{align}
&\relphantom{=}a_{i,r-i}^{\left(\lambda\right)}\left(N+1,x\right)\label{eq:53} \\
& =\left(x+i-r+1\right)\sum_{s=0}^{N-r}\left(-1\right)^{s}\left(N\right)_{s}a_{i,r-i-1}^{\left(\lambda\right)}\left(N-s,x\right)\nonumber\\
  & \relphantom =-i\sum_{s=0}^{N-r}\left(-1\right)^{s}\left(N\right)_{s}a_{i-1,r-i}^{\left(\lambda\right)}\left(N-s,x\right)\nonumber\\
 & \relphantom =+\left(-1\right)^{N-r+1}\left(N\right)_{N-r+1}\left(-1\right)^{i}i!S_{1,i-1}\left(r-i\right)\left(x\right)_{r-i}\nonumber \\
 & =\left(x+i-r+1\right)\sum_{s=0}^{N-r}\left(-1\right)^{s}\left(N\right)_{s}\left(-1\right)^{N-s-1+i-r}i!\nonumber\\
 &\relphantom {=}\times S_{1,i-1}\left(r-1-i\right)\left(N-s-1\right)!H_{N-s-1,r-2}\left(x\right)_{r-1-i}\nonumber \\
 & \relphantom {=}-i\sum_{s=0}^{N-r}\left(-1\right)^{s}\left(N\right)_{s}\left(-1\right)^{N-s+i-r}\left(i-1\right)!S_{1,i-2}\left(r-i\right)\nonumber\\
 &\relphantom{=}\times\left(N-s-1\right)!H_{N-s-1,r-2}\left(x\right)_{r-i}\nonumber\\
 & \relphantom {=}+\left(-1\right)^{N-r+1}\left(N\right)_{N-r+1}\left(-1\right)^{i}i!S_{1,i-1}\left(r-i\right)\left(x\right)_{r-i}\nonumber\\
 & =\left(-1\right)^{N+1+i-r}i!N!\left(x\right)_{r-i}\nonumber\\
 &\relphantom{=}\times\left\{ \left(S_{1,i-1}\left(r-1-i\right)+S_{1,i-2}\left(r-i\right)\right)\sum_{s=0}^{N-r}\frac{H_{N-s-1,r-2}}{N-s}+\frac{S_{1,i-1}\left(r-i\right)}{\left(r-1\right)!}\right\} .\nonumber
\end{align}
 By Lemma \ref{lem:1} and (\ref{eq:53}), we get
\begin{align}
a_{i,r-i}^{\left(\lambda\right)}\left(N+1,x\right) & =\left(-1\right)^{N+1+i-r}i!S_{1,i-1}\left(r-i\right)N!\left(x\right)_{r-i}\label{eq:54}\\
&\relphantom{=}\times\left\{ \frac{H_{N-1,r-2}}{N}+\cdots+\frac{H_{r-1,r-2}}{r}+\frac{H_{r-2,r-2}}{r-1}\right\} \nonumber\\
 & =\left(-1\right)^{N+1+i-r}i!S_{1,i-1}\left(r-i\right)N!H_{N,r-1}\left(x\right)_{r-i}.\qedhere\nonumber
\end{align}

\end{proof}
Therefore, we obtain the following theorem.
\begin{thm}
\label{thm:3} The linear differential equations
\begin{align*}
&\relphantom{=}F^{\left(N\right)}\\
&=\lambda^{N}\left(1+\lambda t\right)^{-N}\\
&\relphantom{=}\times\left(\sum_{r=1}^{N}\sum_{i=0}^{r}a_{i,r-i}^{\left(\lambda\right)}\left(N,x\right)\left(2\lambda+\log\left(1+\lambda t\right)\right)^{-i}\left(\lambda+\log\left(1+\lambda t\right)\right)^{-\left(r-i\right)}\right)F
\end{align*}
has a solution
\[
F=F\left(t;x,\lambda\right)=2\lambda\left(2\lambda+\log\left(1+\lambda t\right)\right)^{-1}\left(1+\lambda^{-1}\log\left(1+\lambda t\right)\right)^{x},
\]
where, for $1\le r\le N$, $0\le i\le r$,
\[
a_{i,r-i}^{\left(\lambda\right)}\left(N,x\right)=\left(-1\right)^{N+i-r}i!S_{1,i-1}\left(r-i\right)\left(N-1\right)!H_{N-1,r-1}\left(x\right)_{r-i}.
\]

\end{thm}
Recall that the $\lambda$-Changhee polynomials, $\Ch_{n,\lambda}\left(x\right)$,
$\left(n\ge0\right)$, are given by the generating function
\begin{align}
F & =F\left(t;x,\lambda\right)\label{eq:55}\\
 & =\frac{2\lambda}{2\lambda+\log\left(1+\lambda t\right)}\left(1+\frac{1}{\lambda}\log\left(1+\lambda t\right)\right)^{x}\nonumber \\
 & =\sum_{n=0}^{\infty}\Ch_{n,\lambda}\left(x\right)\frac{t^{n}}{n!}.\nonumber
\end{align}

Thus, by (\ref{eq:55}), we get
\begin{align}
F^{\left(N\right)} & =\left(\frac{d}{dt}\right)^{N}F\left(t;x,\lambda\right)\label{eq:56}\\
 & =\sum_{k=0}^{\infty}\Ch_{k+N,\lambda}\left(x\right)\frac{t^{k}}{k!}.\nonumber
\end{align}

On the other hand, by Theorem \ref{thm:3}, we get
\begin{align}
 & F^{\left(N\right)}\label{eq:57}\\
 & =\lambda^{N}\sum_{l=0}^{\infty}\left(-1\right)^{l}\binom{N+l-1}{l}\lambda^{l}t^{l}\sum_{r=1}^{N}\sum_{i=0}^{r}a_{i,r-i}^{\left(\lambda\right)}\left(N,x\right)\sum_{m=0}^{\infty}\left(-1\right)^{m}\binom{i+m-1}{m}\nonumber \\
 & =\left(2\lambda\right)^{-i-m}\left(\log\left(1+\lambda t\right)\right)^{m}\sum_{n=0}^{\infty}\left(-1\right)^{n}\binom{r+n-i-1}{n}\lambda^{-\left(r-i\right)-n}\left(\log\left(1+\lambda t\right)\right)^{n}\nonumber \\
 & \relphantom =\times\sum_{s=0}^{\infty}\Ch_{s,\lambda}\left(x\right)\frac{t^{s}}{s!}\nonumber \\
 & =\lambda^{N}\sum_{l=0}^{\infty}\left(-1\right)^{l}\left(N+l-1\right)_{l}\lambda^{l}\frac{t^{l}}{l!}\nonumber\\
 &\relphantom{=}\times\sum_{r=1}^{N}\sum_{i=0}^{r}a_{i,N-i}^{\left(\lambda\right)}\left(N,x\right)\sum_{m=0}^{\infty}\left(-1\right)^{m}\left(i+m-1\right)_{m}\left(2\lambda\right)^{-i-m}\nonumber \\
  &\relphantom{=}\times\sum_{e=0}^{\infty}S_{1}\left(e+m,m\right)\lambda^{e+m}\frac{t^{e+m}}{\left(e+m\right)!}\sum_{n=0}^{\infty}\left(-1\right)^{n}\left(r+n-i-1\right)_{n}\nonumber\\
 &\relphantom{=}\times\lambda^{i-r-n}\sum_{f=0}^{\infty}S_{1}\left(f+n,n\right)\nonumber \\
 & \relphantom =\times\lambda^{f+n}\frac{t^{f+n}}{\left(f+n\right)!}\sum_{s=0}^{\infty}\Ch_{s,\lambda}\left(x\right)\frac{t^{s}}{s!}\nonumber \\
 & =\lambda^{N}\sum_{r=1}^{N}\sum_{i=0}^{r}a_{i,r-i}^{\left(\lambda\right)}\left(N,x\right)\sum_{m=0}^{\infty}\left(-1\right)^{m}\left(i+m-1\right)_{m}\left(2\lambda\right)^{-i-m}\lambda^{m}\frac{t^{m}}{m!}\nonumber \\
 & \relphantom =\times\sum_{n=0}^{\infty}\left(-1\right)^{n}\left(r+n-i-1\right)_{n}\lambda^{i-r-n}\lambda^{n}\frac{t^{n}}{n!}\sum_{l=0}^{\infty}\left(-1\right)^{l}\left(N+l-1\right)_{l}\lambda^{l}\frac{t^{l}}{l!}\nonumber \\
 & \relphantom =\times\sum_{e=0}^{\infty}S_{1}\left(e+m,m\right)\lambda^{e}\frac{e!m!}{\left(e+m\right)!}\frac{t^{e}}{e!}\sum_{f=0}^{\infty}S_{1}\left(f+n,n\right)\lambda^{f}\frac{f!n!}{\left(f+n\right)!}\frac{t^{f}}{f!}\nonumber \\
 & \relphantom =\times\sum_{s=0}^{\infty}\Ch_{s,\lambda}\left(x\right)\frac{t^{s}}{s!}\nonumber\\
 & =\lambda^{N}\sum_{r=1}^{N}\sum_{i=0}^{r}a_{i,r-i}^{\left(\lambda\right)}\left(N,x\right)\sum_{m=0}^{\infty}\left(-1\right)^{m}\nonumber\\
 &\relphantom{=}\times\left(i+m-1\right)_{m}\left(2\lambda\right)^{-i-m}\lambda^{m}\frac{t^{m}}{m!}\sum_{n=0}^{\infty}\left(-1\right)^{n}\left(r+n-i-1\right)_{n}\nonumber \\
 & \relphantom =\times\lambda^{i-r-n}\lambda^{n}\frac{t^{n}}{n!}\sum_{a=0}^{\infty}\left(\sum_{l+e+f+s=a}\left(-1\right)^{l}\lambda^{a-s}\frac{\binom{a}{l,e,f,s}}{\binom{e+m}{m}\binom{f+n}{n}}\left(N+l-1\right)_{l}\right.\nonumber\\
 &\relphantom{=}\left.S_{1}\left(e+m,m\right)S_{1}\left(f+n,n\right)\Ch_{s,\lambda}\left(x\right)\right)\frac{t^{a}}{a!}\nonumber \\
 & =\lambda^{N}\sum_{r=1}^{N}\sum_{i=0}^{r}a_{i,r-i}^{\left(\lambda\right)}\left(N,x\right)\lambda^{-r}2^{-i}\sum_{k=0}^{\infty}\left(\sum_{m+n+a=k}\binom{k}{m,n,a}\left(-\frac{1}{2}\right)^{m}\left(-1\right)^{n}\left(i+m-1\right)_{m}\right.\nonumber \\
 & \relphantom =\times\left(r+n-i-1\right)_{n}\sum_{l+e+f+s=a}\left(-1\right)^{l}\lambda^{a-s}\frac{\binom{a}{l,e,f,s,}}{\binom{e+m}{m}\binom{f+n}{n}}\nonumber\\
 &\relphantom{=}\times \left(N+l-1\right)_{l}S_{1}\left(e+m,m\right)\left. S_{1}\left(f+n,n\right)\Ch_{s,\lambda}\left(x\right)\right)\frac{t^{k}}{k!}\nonumber \\
 & =\sum_{k=0}^{\infty}\left\{ \lambda^{N}\sum_{r=1}^{N}\sum_{i=0}^{r}a_{i,r-i}^{\left(\lambda\right)}\left(N,x\right)\lambda^{-r}2^{-i}\right.\nonumber \\
 & \relphantom =\times\sum_{m+n+a=k}\binom{k}{m,n,a}\left(-\frac{1}{2}\right)^{m}\left(-1\right)^{n}\left(i+m-1\right)_{m}\left(r+n-i-1\right)_{n}\nonumber \\
 & \relphantom =\times\sum_{l+e+f+s=a}\left(-1\right)^{l}\lambda^{a-s}\frac{\binom{a}{l,e,f,s}}{\binom{e+m}{m}\binom{f+n}{n}}\left(N+l-1\right)_{l}\nonumber \\
 & \relphantom =\left.\times S_{1}\left(e+m,m\right)S_{1}\left(f+n,n\right)\Ch_{s,\lambda}\left(x\right)\right\} \frac{t^{k}}{k!}.\nonumber
\end{align}

Therefore, by comparing the coefficients on both sides of (\ref{eq:56})
and (\ref{eq:57}), we obtain the following theorem.
\begin{thm}
\label{thm:4} For $N\in\mathbb{N},$ and $k\in\mathbb{N}\cup\left\{ 0\right\} $,
we have
\begin{align*}
 & \Ch_{k+N,\lambda}\left(x\right)\\
 & =\lambda^{N}\sum_{r=1}^{N}\sum_{i=0}^{r}a_{i,r-i}^{\left(\lambda\right)}\left(N,x\right)\lambda^{-r}2^{-i}\sum_{m+n+a=k}\binom{k}{m,n,a}\\
 & \relphantom =\times\left(-\frac{1}{2}\right)^{m}\left(-1\right)^{n}\left(i+m-1\right)_{m}\left(r+n-i-1\right)_{n}\\
 & \relphantom =\times\sum_{l+e+f+s=a}\left(-1\right)^{l}\lambda^{a-s}\frac{\binom{a}{l,e,f,s}}{\binom{e+m}{m}\binom{f+n}{n}}\\
 & \relphantom =\times\left(N+l-1\right)_{l}S_{1}\left(e+m,m\right)S_{1}\left(f+n,n\right)\Ch_{s,\lambda}\left(x\right),
\end{align*}
where, for $1\le r\le N$, $0\le i\le r$,
\[
a_{i,r-i}^{\left(\lambda\right)}\left(N,x\right)=\left(-1\right)^{N+i-r}i!S_{1,i-1}\left(r-i\right)\left(N-1\right)!H_{N-1,r-1}\left(x\right)_{r-i}.
\]

\end{thm}
\bibliographystyle{amsplain}
\providecommand{\bysame}{\leavevmode\hbox to3em{\hrulefill}\thinspace}
\providecommand{\MR}{\relax\ifhmode\unskip\space\fi MR }
\providecommand{\MRhref}[2]{%
  \href{http://www.ams.org/mathscinet-getitem?mr=#1}{#2}
}
\providecommand{\href}[2]{#2}

\end{document}